\documentclass[a4paper]{article}
\usepackage{fullpage,url,amsmath,amsfonts,amssymb, tedmath}

\newcommand{\D}{\mathcal{D}}

\title{Convolutional codes  
 from unit schemes.}
\author{Ted
 Hurley\footnote{National Universiy of Ireland Galway, email:
 Ted.Hurley@NuiGalway.ie}}
\date{} 
\begin{document}
\maketitle
\begin{abstract}\let\thefootnote\relax
Convolutional codes are constructed, designed and analysed using row and/or block structures of unit algebraic schemes. 
Infinite series of such codes and
of codes with specific properties are derived.  
Properties are  shown  
 algebraically and algebraic decoding methods are derived. 
For a given rate and given error-correction 
 capability at each component, convolutional codes with these
 specifications and with efficient decoding algorithms are 
 constructed. Explicit prototype examples are given but in general large lengths and large error capability are achievable. 
Convolutional codes with efficient decoding algorithms at or near the maximum free distances attainable for the parameters 
are  constructible. Unit memory convolutional
 codes of maximum possible free distance are designed with practical 
 algebraic decoding algorithms.  

LDPC (low
density parity check) convolutional codes with efficient  decoding
 schemes are constructed and analysed by the methods.  
Self-dual and dual-containing convolutional codes may also be designed by the methods; dual-containing codes enables the construction of quantum codes.  
\end{abstract}

\section{Introduction}
Convolutional codes are error-correcting codes which 
are used extensively in many  applications 
including digital video,
radio, mobile communication, and satellite/space communications. 

  Background on  convolutional codes  
  may be found in \cite{blahut}, 
  \cite{joh}, \cite{mceliece},  \cite{mac}   
 and many others.  
The oft quoted theory by McEliece, \cite{mac}, 
points out  
the lack of algebraic methods for the  construction of convolutional
codes and notes that the existing ones
had been found by search methods which severely  limit their size and
availability.  
Here algebraic formulations using row and/or blocks of 
unit schemes are developed.  
A major advantage is that the control matrix and 
the finite right inverse, when it exists, of the generator matrix 
can be obtained directly from the construction.  The form of the control matrix 
leads to implementable  error-correcting algorithms. 
Efficient decoding techniques may
thus be derived. When the errors have been discovered and eliminated the right inverse may then be used to recover the required information vector. 
 Explicit prototype examples are given; in general
large lengths and large error capability are achievable. 


 


Using the techniques of \cite{hurleyunit} and \cite{hurley}, convolutional codes with
maximum distance and other required properties may be constructed. For a given
rate and required error-correction at each component, convolutional
codes with these specifications are constructed and these have 
 algebraic efficient decoding algorithms. This gives types of Shannon 
achieving convolutional codes.  

Using particular types unit schemes enables convolutional codes to be constructed with specific properties. 
Using schemes where one of the units has low density leads to   the
construction of LDPC (low density parity check) convolutional codes,
Section \ref{ldpc}. 
Specialising to 
orthogonal schemes leads to the construction of series of self-dual and
dual-containing convolutional codes, Section \ref{selfdual}. 
 Dual-containing  (including  self-dual) 
codes lead to  the construction of quantum codes,
\cite{calderbank}; see also \cite{quantum}. 

 
Different  {\em equivalent}   
 definitions of  convolutional codes are given in the literature and
 these are analysed very nicely and comprehensively in 
\cite{rosenthal1}. Algebraic definitions
 and formulations are of interest here.   

A rate $\frac{k}{n}$ convolutional code with parameters $(n,k,\de)$ over a field $\F$ is 
 a submodule of $\F[z]^n$ generated by a reduced
basic matrix $G[z] =(g_{ij}) \in \F[z]^{r\ti n}$ of rank $r$,  where $n$ is
the {\em length},  $\de = \sum_{i=1}^k \de_i$ is the {\em degree}
with  $\de_i= \max_{1\leq j\leq k}{\deg g_{ij}}$. The  
$\mu=\max_{1\leq i\leq r}{\de_i}$ is known as the {\em memory} of the
code and the code may then be given with parameters $(n,k,\de;\mu)$. 
The parameters $(n, r,\delta;\mu, d_f)$ are  used for such a code 
with free (minimum) distance $d_f$.

A convolutional code may equivalently be described as
follows. A 
convolutional code $\C$ 
of length $n$ and dimension $k$ is a direct
summand of $\F[z]^n$ of rank $k$. See for example \cite{heine} and
\cite{rosenthal1}. Here $\F[z]$ is the polynomial
ring over $\F$ and  $\F[z]^n = \{(v_1, v_2, \ldots,
v_n) : v_i \in \F[z]\}$. 

 
Suppose $\C$ is a convolutional code in $\F[z]^n$ of rank $k$. A generating matrix $G[z] \in
\F[z]_{k\times n}$ of $\C$ having rank $k$ is called a
{\em generator} or {\em encoder matrix}  of $\C$. 
A matrix $H \in \F[z]_{n\times(n-k)}$ satisfying $\C = \ker H =
\{v \in \F[z]^n : vH = 0 \}$ is said to be a {\em control matrix} or
    {\em check matrix} of the code $\C$. 


For a given generator matrix 
$G[z] \in \F[z]_{r\ti n}$ of a convolutional code,  a {\em codeword} is of
the form  $u(z)G[z]$ with  $u(z) \in \F[z]^r$. 
Here   $u(z)$ is  {\em the information vector } of  the
codeword $u(z)G[z]$. 
Now  $u(z) = \sum_i\al_iz^i$ for vectors $\al_i\in \F^r$; the $\al_i$ are 
  the {\em components} of
$u(z)$. The {\em support} of $u(z)$ is the
number of non-zero (vector) components of $u(z)$. What is required is
the recovery of the information vector when it has been encoded,  by
$u(z) \mapsto u(z)G[z]]$,  and then  
transmitted over a `noisy' channel.


The papers  \cite{heine1}, \cite{rosenthal} introduce  
certain algebraic decoding techniques  for convolutional
codes; the decoding techniques here are different. Vetterbi or
sequential decoding are available for convolutional codes, see
\cite{blahut}, \cite{joh} or \cite{mceliece} and references therein. 


\quad 

 The maximum distance attainable by an $(n,r)$
linear code is $(n-r+1)$ and this is known as the {\em Singleton bound}, 
 see \cite{blahut} or \cite{mac}. A linear code
$(n,r)$ attaining the maximum distance possible  
is called an {\em mds} (maximum
distance separable) code. 
By Rosenthal and Smarandache, \cite{ros},
 the maximum free distance attainable by an
$(n,r,\delta)$ code is  $(n-r)(\floor{\de/r}+1)+
\de +1$. The case  $\delta
=0$, which is the case of zero memory, corresponds to the linear 
Singleton bound $(n-r+1)$. 
The bound $(n-r)(\floor{\de/r}+1)+
\de +1$  
is then called  the {\em generalised Singleton  bound}, \cite{ros}, GSB,
and a convolutional code attaining this bound is known as  an {\em mds
convolutional code}.

In \cite{ros}, Rosenthal and Smarandache give a   a non-constructive
proof of the existence of
$(n,r,\de)$ mds convolutional  codes 
and then in
\cite{smar} explicit examples over suitably large fields are
given. The papers \cite{ros} and \cite{smar} 
are fundamental and major contributions to the area.   

The paper \cite{napp} of Napp and Smarandache  relates MDP  
(maximum distance profile) convolutional codes and constructs such
codes. In addition \cite{chan} discusses  methods for constructing unit
memory mds convolutional codes.

Here memory 1 GSB convolutional codes are designed with efficient
 decoding algorithms. For a given rate and a given $t$ these can be designed with this rate and with the property that $t$ errors may be corrected at
 each component of an error vector.

Many of the constructions in the literature are special cases of the
general construction here; those in \cite{blahut} and
\cite{mceliece,mac} are obtained as  special small cases of the general
constructions here.  



\subsection{Material Organisation}
Section \ref{rows} describes the general methods and proves related
results.  Prototype initial constructions are given in Section 
\ref{examples1}; this section could be consulted to see the range and 
 types of constructions possible. 
  Algebraic algorithmic decoding methods  
are derived  in Section \ref{decoding}.  
Section \ref{cheb1} discusses convolutional codes obtained from unit
matrices where rows of the matrix taken in succession or in certain
arithmetic sequences form mds linear codes; see \cite{hurleyunit} for
the linear analogue of these.   
 Prototype examples from these unit schemes are given in 
Section \ref{design}. 

The constructions in
Sections \ref{examples1} and \ref{design}   
 can also serve as prototypes for longer length schemes which are 
available from the general methods. 



Section 
\ref{blocks} generalises the row structure design method to a 
block structure design method. 
Section  \ref{ldpc} describes methods for designing and analysing    
LDPC (low density parity check) 
convolutional codes. Section \ref{selfdual} deals
with the design and properties of   
 self-dual and dual-containing convolutional codes. Self-dual and
 dual-containing codes have been used for constructing quantum codes.  

Some of the block cases in Section \ref{blocks}  
overlap some of those in \cite{jessica}.  


\subsection{Algebraic background}\label{Set-up}
Most of the required background may be found in \cite{blahut}
and further background is to be found in \cite{seh}.
 
 Methods are given in  \cite{hur1} and  \cite{hur0} for constructing 
unit-derived (linear) codes;  further details including the required algebra may be found 
 in expanded book chapter form at \cite{hur2}.  The unit-derived method
 may be described briefly  as follows.
Let $R_{n\ti n}$ denote the ring of
 $n\ti n$ matrices with entries from $R$, a ring with identity. 
Suppose  $UV=1$ in $R_{n\ti n}$.   
Taking any $r$ rows of $U$
as a generator matrix defines an $(n,r)$ code  and the check
matrix is obtained by deleting the corresponding columns of $V$.
Let such a code be denoted by $\mathcal{C}_r$.
 When $R$ is a
field and $U$ is a type of Vandermonde or Fourier matrix, then taking
consecutive rows of $U$ or rows of $U$ in arithmetic sequence, with arithmetic difference $k$ satisfying $\gcd(n,k)=1$, results in
mds linear codes, see \cite{hurleyunit}. When $R$ is a field and  the matrix $V$ has the property that the
determinant of any square 
submatrix of $V$ is non-zero then any such code $\mathcal{C}_r$ is an mds
$(n,r,n-r+1)$ linear code, see \cite{hurley} for details. These methods
 are extended here to obtain good performing convolutional codes

 Convolutional codes have been constructed in \cite{hur10} using $UV=1$ in
 which $\{U,V\}$ are  themselves { polynomial  or  group rings}; further
 details on these may be obtained in book chapter form in  
 \cite{hur11}. 

Note: The expression $\un{0}$ is used to denote a vector or matrix consisting of
all zeros; so as to avoid complicating the notation, 
 the  size 
of $\un{0}$ is determined from  the context in which it is used.

\section{Convolutional codes from unit schemes: general construction}\label{rows}
 
A general method is described for constructing convolutional codes from
unit schemes.  

Suppose that 
$UV=1$ in $\F_{n\ti n}$. Denote the rows of $U$ by
$\{e_0,e_1, \ldots , e_{n-1}\}$ and the columns of $V$ by $\{f_0,f_1,
\ldots, f_{n-1}\}$. Then $e_if_j = \delta_{ij}$, the Kronecker
delta. 

Now form $r\ti n$ matrices $E_i$ for $i=0, \ldots, s$ where each $E_i$ consists
of choosing $k$ rows of $U$ for $1\leq k \leq r$ together with
$r-k$ rows consisting of the 
 $1\ti n$ zero vector; it is also assumed that $E_i\neq E_j$ for $i \neq j$. 
Define 
\begin{equation}\label{eq:geneq}
G[z]= E_0+E_1z+\ldots + E_sz^s .  
\end{equation}



Then  $G[z]$ is  a generating matrix for a convolutional
code. 

Many existing convolutional codes are equivalent to codes formed in this way.
Not every choice of the  
 $E_i$ will be suitable or useful. 
It is desirable to choose the $E_i$ so that $G[z]$ is {\em
noncatastrophic}.  Also $G[z]$ is required to have rank $r$. 

For the definition of 
{\em noncatastrophic} see for example \cite{blahut} or \cite{mac}. Here it is sufficient to note, see \cite{mac}, that $G[z]$ is
noncatastrophic if it has a polynomial right inverse, that is, if there exists a
polynomial $n\ti r$ matrix $H[z]$ of rank $r$ such that
$G[z]H[z]=I_r$, the identity $r\ti r$ matrix. Thus it is desirable 
 but not necessary that the rows used in $E_0$ be linearly independent.

For a  given $n\ti n$ system $UV=1$ and a given $r < n$ the method 
allows the construction of codes of rate  $\frac{r}{n}$.  
A big advantage of this scheme is that the control matrix and the finite right
inverse, when it exists,  of this generator matrix may directly be
constructed from the columns of the inverse $V$ of $U$; this can lead to decoding algorithms.  


Suppose $U= \begin{pmatrix} e_0 \\ e_1 \\ \vdots \\ e_{n-1} \end{pmatrix}$
and $V=(f_0, f_1, \ldots, f_{n-1})$ are the row and column structures of
$U,V$ respectively. Then $e_if_j = \de_{ij}$. Each $E_i$  is made up  of rows of $U$ and say $e_j\in E_i$ if $E_i$
contains the row $e_j$  and $e_j\not \in E_i$
to mean that  $E_i$ does not contain the row $e_j$. 
Define $E_j^*$ to be the transpose of $E_j$ with $e_i\T$ replaced by
$f_i$. For example if $E=\begin{pmatrix} e_0 \\ e_1 \\ e_3 \end{pmatrix}$
then $E^*= (f_0,f_1,f_3)$.  

Because $e_if_j = \de_{ij}$ the construction of the right inverse of $G[z]$, when
it exists, and the construction of the control matrix for the
convolutional code generated by $G[z]$ becomes
straightforward. Knowing the control matrix often leads 
to an algebraic decoding technique. 

Assume now that $E_0$ has rank $r$; this is not a great
restriction since in any case it is required that  $\rank (G[z]) = r$. Hence the zero
vector does not occur in the construction of $E_0$. 
When $E_0$ is formed from rows 
$e_i$ which are not contained in any of
 the other $E_i$, for $ i\geq 1$, 
then the resulting matrix $G[z]$ in (\ref{eq:geneq}) is easily seen to be 
non-catastrophic.

\begin{lemma}\label{jerome} Suppose  $G[z]=E_0+E_1z+\ldots +
 E_sz^r$, as in equation (\ref{eq:geneq}), and 
that $e_j\in E_0$ implies
 $e_j\not \in E_i$ for $i\neq 0$. Then $G[z]$ is a noncatastrophic
 generator matrix for a convolutional $(n,r)$ code.  
\end{lemma}  
\begin{proof} By \cite{mac} if there exists a polynomial matrix $H[z]$ such
 that $G[z]H[z]= I_r$, the identity $r\ti r $ matrix, then $G[z]$ is
 noncatastrophic and it must then have rank $r$. Now $e_if_j=\delta_{ij}$
 and thus $E_iE_0^* =
 \delta_{i0}I_r$ as no row of $E_i$ is a row of $E_0$ for $i\neq 0$. It 
 then follows that $G[z]E_0^*=I_r$.
\end{proof}

The above is a simple case but many other cases exist where the
$G[z]$ constructed is 
noncatastrophic; see the `principles' adopted in Section \ref{catcat} below. 
In these cases the control matrix may also be constructed in
short order. 

\subsection{Design to specific parameters} The general method allows the design of convolutional codes to specific requirements.  To design a $(n,r,\de;\mu)$
convolutional code, find a suitable $n\ti n$ system $U,V$ with $UV=1$ and then  choose
appropriate $E_0,E_1, \ldots, E_\mu$ where each $E_i$ are $r\ti n$ and consist of rows of $U$ with  possible zero rows. 

Below we are more specific in the designs to achieve other objectives.
Most of the
convolutional codes in use and available have been constructed on a
case by case basis and/or by computer; computer techniques are now
often beyond the range of computers. Codes in use can be shown to be equivalent to special cases of the design methods here. 

See Section \ref{specify} for specific design methods using Fourier and/or Vandermonde unit schemes.
These have very good distance properties and efficient error-correcting algorithms.

Matrices, necessarily invertible,  with additional 
properties such as being  
unitary or orthogonal  may also be chosen  to design  suitable 
 convolutional codes with specific properties (such as for example 
being {\em LDPC} or dual-containing); see Sections \ref{blocks} and \ref{selfdual} for these developments.


\subsection{Design noncatastrophic generator matrices}\label{catcat}
In designing such a $G[z]$ as in equation (\ref{eq:geneq}),  
 it is desirable that
$G[z]$ be noncatastrophic. In so designing  the generator 
matrix constructed using formula (\ref{eq:geneq}),   the following principles are helpful. 
\begin{itemize} \item\label{level} If a row occurs in both
$E_j$ and $E_j$ for $i\neq j$ then it should not occur at the same
`level' in both of these, that is if the row $e_k$ occurs as the
$s^{th}$ row of $E_i$ then it should not occur as the $s^{th}$ row of
$E_j$.   \item Some row of $E_0$ should not occur in any other
  $E_j$.  Although this is not always necessary, it is desirable.
\item If $e_i$ occurs in $s^{th}$ row of $E_k$ and $e_j$ occurs in
  $t^{th}$ row of $E_k$ then $e_i$ should not occur in $t^{th}$ row of
  some other $E_w$ whenever $e_j$ occurs in $s^{th}$ row of $E_w$; in
  other words if $e_i,e_j$ occur in some $E_k$ then they don't occur
  at switched around level in some other $E_w$.  
\item It is also helpful that some row in the coefficient of the highest power of $z$ does not occur in any other $E_j$. 
\end{itemize} 


\subsubsection{Noncatastrophic versus catastrophic}
A slight change in the ordering of the rows in the construction can make
a major difference to the generator matrix; this is demonstrated by the following examples.   


Let $UV=I_4$ in a field $\F$ where  $U$ has rows
   $\{e_0,e_1,e_2,e_3\}$ in order and $V$ has columns in order
   $\{f_0,f_1,f_2,f_3\} $ in order.
Define $G[z]= \begin{pmatrix}e_0 \\ e_1 \\ e_2  \end{pmatrix}+
\begin{pmatrix}e_1 \\ e_2 \\ e_3 \end{pmatrix}z $. Then $G[z]$ is a
     noncatastrophic matrix with right polynomial 
 inverse $K[z]= (f_0,f_1,f_2)-
     (\un{0},f_0,f_1)z + (\un{0},\un{0},f_0)z^2$. 
 Let $H[z] =
 f_3-f_2z+f_1z^2-f_0z^3$. It is easily checked that $G[z]H[z]=0_{3\ti 1}$ and
 thus $H[z]$ is a control matrix for the convolutional $(4,3,3)$ code
 generated by $G[z]$. The free distance of this code depends on the
 distances of linear codes generated by taking rows of $U$. 

 Now let
$K[z]= \begin{pmatrix}e_0 \\ e_1 \\ e_2\end{pmatrix}+
\begin{pmatrix}e_3\\ e_1 \\ e_2 \end{pmatrix}z $ and then $K[z]$ is a
     catastrophic matrix  as $L[z]=(f_0,f_1,f_2) -
     (\un{0},f_1,f_2)z + (\un{0},f_1,f_2)z^2 - (\un{0},f_1,f_2)z^3+
     \ldots \ldots $ is its right inverse and  
 $K[z]$ does not have a polynomial
     right inverse. 

Thus a small change in the order of the chosen rows changes the
generator matrix from being noncatastrophic to
     being catastrophic. Notice in $K[z]$ that 
     the first `principle' in 
     \ref{level} above is violated. 


\subsection{Permuting rows} 
Permuting the rows in equation (\ref{eq:geneq}) gives generator
matrices for {\em different} convolutional codes. 
Some of these are equivalent but not many. 
Different convolutional codes even of the same rate 
are obtainable from the same unit scheme.

\subsection{Explicit prototype examples}\label{examples1}
Prototype examples are given. Those presented here are necessarily  of
relatively small length although large  length and large error-capability are
attainable in practice.     
Examples \ref{item00} and \ref{item0}  below demonstrate how codes  in
the literature occur as special cases of the general construction. 
\begin{enumerate}
\item\label{item13} Let   
$U,V\in \F_{5\ti 5}$ satisfy  $UV=I_5$ where $U=\begin{pmatrix} e_0 \\
 e_1 \\ e_2 \\ e_3 \\ e_4  \end{pmatrix}$  and  
 $V=(f_0,f_1,f_2,f_3, f_4) $ are the row and column structures of
$U,V$ respectively.  
Let $G[z]= \begin{pmatrix}e_0 \\ e_1 \\ e_2\end{pmatrix}+
\begin{pmatrix}e_1\\ e_2 \\ e_3 \end{pmatrix}z +\begin{pmatrix}e_2 \\ e_3
\\ e_4 \end{pmatrix}z^2 $.  Then $G[z]$ is a generator matrix
 of a $(5,3,6)$ convolutional code. It is seen directly 
that $G[z]\{(f_0,f_1,f_2)-(\un{0},f_0,f_1)z\}= I_{3\times 3}$ and thus
$G[z]$ is a noncatastrophic generator matrix of rank $3$.  
  
 Let $H[z] = (f_3,f_4)- (f_2, f_3)z + (f_0,f_1)z^3 - (\un{0},f_0)z^4 $. 
 Then $G[z]H[z]=0_{3\times 2}$ and so
     $H[z]$ is a control $5\ti 2$ matrix for the convolutional code. 

The degree of the code is $6$ and the rank is $3$
     so the GBS for such a code is $(5-3)(3)+6+ 1=13$ by \cite{ros}. Now
     $U$ may be chosen so that $(e_0,e_1,e_2), (e_1,e_2,e_3),
     (e_2,e_3,e_4)$ generate mds linear $(5,3,3)$ codes; this may be
     done by for example choosing $U$ to be a Fourier $5\ti 5$ matrix
     over a suitable field, see \cite{hurleyunit}.  A suitable field is
     $GF(11) = \Z_{11}$ or any field which has a primitive $5^{th}$ root
     of unity.  

The nature of the control matrix is such that it can correct up to one error
     at each vector component of a polynomial  error vector. 
\item The method in item \ref{item13} may be generalised to give $(n,r)$ convolutional codes which can correct up to $\floor{\frac{n-r}{2}}$ errors at each vector component of an error polynomial vector. 


\item\label{item00} Consider $U=\begin{pmatrix} 1 & 1 \\ 1&0 \end{pmatrix}, V= \begin{pmatrix} 0& 1 \\ 1&1 \end{pmatrix}$ in $GF(2) = \Z_2$. Then $UV= I$.  Denote the rows of $U$ in order by $e_0, e_1$ and the columns of $V$ in order by $f_0,f_1$. Define $G[z] = e_0 + e_1 z$. Let $K[z] = f_0$ and then $G[z]K[z] = 1$.  Thus $G[z]$ is a noncatastrophic generator matrix for a convolutional code. Let $H[z] = f_1 - f_0z$. Then $G[z]H[z] = 0_{1\ti 1}$ and hence $H[z]$ is a control matrix of the code. The free distance is $3$. 

Now   
     let $G[z]= e_0+e_1z+e_0z^2$. The  convolutional code generated
     by $G[z]$ is  exactly the same as that  given on  page 276 of  Blahut's book
     \cite{blahut} as  $G[z] = (1+z+z^2, 1+z^2)$. It is
     easily seen that $G[z]$ is 
     catastrophic directly from the form given here: Let  $H[z] = f_0-f_0z^2 + f_0z^4 - \ldots $  and 
then  $G[z]H[z]= 1$.  $G[z]$ has no polynomial right
inverse.  
\item\label{item0} Consider $U=\begin{pmatrix} 0 & 1 & 0 \\ 1 & 0 & 1 \\ 1 & 0 & 0
		  \end{pmatrix}$ over $Z_2=GF(2)$. Now $U$ is
      invertible. Denote the rows, in order, of $U$ by $\{e_0,e_1,e_2\}$ and
      the columns, in order, 
      of the inverse $V$ of $U $ by $\{f_0,f_1,f_2\}$. Define $G[z] =
      \begin{pmatrix} e_0 \\ e_1 \end{pmatrix} + \begin{pmatrix}e_2 \\
      \un{0} \end{pmatrix}z + \begin{pmatrix} \un{0} \\ e_0
			      \end{pmatrix}z^2$ . 
Now $G[z] = \begin{pmatrix} z& 1& 0 \\ 1 & z^2 & 1 \end{pmatrix}$ and
      thus $G[z]$ is the same as that constructed on page 279 of
     Blahut,  \cite{blahut}.

In the form given it is easy to find the right inverse for $G[z]$ and
      the control matrix for the convolutional $(3,2,4)$ 
code generated by $G[z]$.
Let $K[z]= (f_0,f_1)- (f_1,\un{0})z^2$ and then $G[z]K[z] = I_{2 \ti
  2}$. Letting  $H[z] = f_2 -f_0z + f_1z^2 $ then $G[z]H[z] = 0_{2\ti 1}$. 


\item\label{item1} 
Suppose  
$U,V\in F_{3\ti 3}$ satisfies  $UV=I_3$ where $U=\begin{pmatrix} e_0 \\ 
 e_1 \\ e_2  \end{pmatrix}$  and  
 $V=(f_0,f_1,f_2) $ are the row and column structures of
$U,V$ respectively.  
Consider $G[z]= e_0 + e_1z + e_2z^2$ . 
 By direct verification 
 $G[z]*f_0 = I_{1\times 1}$ and thus
$G[z]$ is a noncatastrophic matrix of rank $1$ of a $(3,1,2)$
      convolutional code $\mathcal{C}$. Also
$G[z]*\{(f_1,f_2)-(f_0,\un{0})z - (\un{0},f_0)z^2\} = 0_{1\ti 2}$ 
	where $\un{0}$ is the $3\ti 1$ zero matrix gives the control
      matrix  $H=(f_1,f_2)-(f_0,\un{0})z - (\un{0},f_0)z^2$ of the
      code. 

The maximum free
distance, GSB, of a $(3,1,2)$ code is, by \cite{ros}, $9$.
Suppose now $U$ is
a matrix over $F$ such that any $r$ rows for $1\leq r \leq 3$ generate
an mds linear code, that is generate a $(3,r,3-r+1)$ linear code. Such
matrices exist but depend on the field $\F$.  
In fact  this code attains the free GSB 
distance $9$ and is an MDS convolutional code; this may be verified directly by
      looking at $f[z]*G[z]$ where $f[z]$ is a polynomial. This is 
      a special case of theorem \ref{general1} below.



\item\label{item2} 
Suppose  
$U,V\in F_{3\ti 3}$ are as in example \ref{item1}. 
Let $G[z]= \begin{pmatrix}e_0 \\ e_1 \end{pmatrix}+
\begin{pmatrix}e_1\\ e_2 \end{pmatrix}z  = E_0+E_1z$, say.  Then $G[z]$ is a noncatastrophic generator matrix
of rank $2$ of a $(3,2,2)$ convolutional code $\mathcal{C}$. By direct verification 
 $G[z]*\{(f_0,f_1)-(\un{0},f_0)z \} = I_{2\times 2}$ and thus
$G[z]$ is a noncatastrophic matrix of rank $2$; here $\un{0}$ is the
$3\times 1$ zero matrix. 
 The $3\times 1$ (control) matrix $H[z]$ of rank $1$   
such that  $G[z]H[z]=0_{2\times 1}$ 
 is  $H[z] = f_2-f_1z+f_0z^2$.

The maximum free
distance of a $(3,2,2)$ code is  $5$ by \cite{ros}. 
Suppose now $U$ is
a matrix over $F$ such that any $r$ rows for $1\leq r \leq 3$ generate
an mds linear code, that is generate a $(3,r,3-r+1)$ linear code. Such
matrices exist but depend on the field $\F$ ,see \cite{hurleyunit} and below. 

Define now  $G[z]= \begin{pmatrix}e_0 \\ e_1 \end{pmatrix}+
\begin{pmatrix}e_1\\ -e_2 \end{pmatrix}z$ in a field of characteristic
$\neq 2$ where $U$ is the Fourier $3\ti 3$ matrix. This gives a GSB  
 convolutional  code $(3,2,2;1,5)$. 
\item\label{example22} Let $U,V \in F_{3\ti 3}$ be as in
  \ref{item1}. Define $G[z]= \begin{pmatrix}e_0 \\ e_1 \end{pmatrix}+
\begin{pmatrix}e_1\\ e_2 \end{pmatrix}z + \begin{pmatrix} e_2
  \\ e_3 \end{pmatrix}z^2$. Then $G[z]\{(f_0,f_1) - (\un{0},f_0)\} =
I_2$ and so $G[z]$ is a noncatastrophic generator matrix of a
$(3,2,4)$ convolutional code. The control matrix is $H[z]=(f_2,f_3) -
(f_1,\un{0})z -(\un{0},f_1)z^2 + (\un{0},f_0)z^3$. 
\item 
Matrices in  examples \ref{item2} and \ref{item1} may be given precisely as follows. Let $U= \begin{pmatrix}
1 & 1 & 1 \\ 1 & \om & \om^2 \\ 1&\om^2 & \om \end{pmatrix}
= \begin{pmatrix} e_0 \\ e_1 \\ e_2 \end{pmatrix}$ where
$\om$ is a primitive 3rd root of unity in a field $F$ in which the 3rd
root of unity exists. 

The  fields $\F$ that may be used  
include $GF(2^2),  GF(7)$. 
Then $G[z] = \begin{pmatrix} e_0 \\ e_1 \end{pmatrix}
+ \begin{pmatrix} e_1 \\ e_2\end{pmatrix}z$ is a noncatastrophic
  generating matrix for a $(3,2,2)$ convolutional code which has free
  distance $\geq 4$ and such that the distance of $f[z]G[z]$,
  where $f[z]$ is a polynomial of $1\ti 3$ vectors of support $t$, is at least
  $4+(t-1)$. When the field has characteristic $\neq 2$ then
$G[z] = \begin{pmatrix} e_0 \\ e_1 \end{pmatrix}
+ \begin{pmatrix} e_1 \\ -e_2\end{pmatrix}z$ is a noncatastrophic
 generator matrix of a $(3,2,2)$ convolutional code of free distance
 $5$ and such that the free distance of $f[z]G[z]$,
  where $f[z]$ is a polynomial $1\ti 3$ of support $t$, is at least
  $4+(t)$.

For example when $\F=GF(7)$, and noting that $2$ is a primitive 3rd root
of unity in $GF(7)$, we obtain the GSB $(3,2,2)$ code with
generator matrix $G[z] = \begin{pmatrix} 1 & 1 & 1 \\ 1 & 2 &
  4 \end{pmatrix} + \begin{pmatrix} 1& 2 & 4 \\ -1 & -4 &
  -2 \end{pmatrix}z$ over $GF(7)$. 


\item\label{example5} Suppose now $U,V$ are as in 
example \ref{item13}. 
Let  $G[z]= \begin{pmatrix}e_0 \\ e_1 \\ e_2 \end{pmatrix}+
\begin{pmatrix}e_2 \\ e_3
\\ e_4 \end{pmatrix}z $.  Then $G[z]$ is a noncatastrophic generator matrix
of rank $3$ of a $(5,3,3)$ convolutional code . We may verify directly 
that $G[z]*\{(f_0,f_1,f_2)-(\un{0},\un{0},f_0)z\}= I_{3\times 3}$ and thus
$G[z]$ is a noncatastrophic matrix of rank $3$; here $\un{0}$ is the
$5\times 1$ zero matrix. Also
      $G[z]\{(f_3,f_4)-(f_1,f_2)z+(\un{0},f_0)z^2\}= 0_{3\ti 2}$ then
      gives the control matrix.

      The GSB for a $(5,3,3)$ convolutional code is by
      \cite{ros} 2(2)+3+1 = 8. However $\al_0G[z]$ may have  distance
      $6$ for a $1\ti 3$ vector $\al_0\in F^{1\ti 3}$.  
It may be shown that $f[z] G[z]$ has distance $6+(t-1)$ where $t$ is
      the support of $f[z]$. 



\end{enumerate}  

High rate $\frac{(n-1)}{n}$ convolutional codes are presented from the general methods. These are shown to have nice algebraic decoding
techniques. 

\begin{enumerate}\item\label{high1}
Let $UV=1$ in $F_{4\ti 4}$ where $U$ has rows $e_0,e_1,e_2,e_3$ and
$V$ has columns $f_0,f_1,f_2,f_3$.
Define $G[z]= \begin{pmatrix} e_0 \\ e_1 \\ e_2 \end{pmatrix}
+ \begin{pmatrix} e_2 \\ e_3 \\ e_1 \end{pmatrix}z$.

The right inverse $K[z] = (f_0,f_1,f_2) - (\un{0}, f_2,f_0)z +
(\un{0},f_0,\un{0})z^2$ is calculated and shows that $G[z]$ is a
noncatastrophic generator for a $(4,3,3)$ convolutional code. The
control matrix is $H[z] = f_3 - f_1z+f_2z^2-f_0z^3$ with $G[z]H[z]=
0$. 

The GBS for a $(4,3,3)$ convolutional code is, by \cite{ros},
$6$. Provided $\begin{pmatrix} e_0 \\ e_1 \\ e_2 \end{pmatrix}$
and $\begin{pmatrix} e_2 \\ e_3 \\ e_1 \end{pmatrix}$ generate
$(4,3,2)$ linear codes this convolutional code has free distance
$4$. However it will correct errors of the form $\al z^i$ with $\al\in
F^4$ and errors of the form $\al[z] = \al_0 z^{i_0} + \al_1z^{i_1} + \ldots + \al_rz^{i_r}$,
where $i_{j+1} \geq i_j + 4 $ for $0 \leq j \leq r-1$, and $\al_i\in
F^4$.

\item\label{high2} Let $U,V$ be as in construction \ref{high1}. above. 
Define  $G[z]= \begin{pmatrix} e_0 \\ e_1 \\ e_2 \end{pmatrix}
+ \begin{pmatrix} e_3 \\ e_2 \\ \un{0} \end{pmatrix}z$.
Then $K[z]= (f_0,f_1,f_2)-(\un{0},\un{0},f_1)z$ satisfies
$G[z]K[z]=I_3$ and shows that $G[z]$ is a noncatastrophic generator
for a $(4,3,2)$ convolutional code. The GSB for such a code is, by \cite{ros}, 4. The
control matrix is $H[z]= f_3 -f_0z$. The code can correct one error
provided $(e_1,e_2)$ generates a $(4,2,3)$ linear code. It can also
correct errors in each coefficient of an error-polynomial in which the
degrees are at least two apart. 
\item\label{high3} Let  
Let $UV=1$ in $F_{5\ti 5}$ where $U$ has rows $e_0,e_1,e_2,e_3,e_4$ and
$V$ has columns $f_0,f_1,f_2,f_3,f_4$.
Define $G[z]= \begin{pmatrix} e_0 \\ e_1 \\ e_2 \\ e_3 \end{pmatrix}
+ \begin{pmatrix} e_3 \\ e_4 \\ e_1 \\ e_2 \end{pmatrix}z$.
The right inverse for $G[z]$ is $K[z] =
(f_0,f_1,f_2,f_3)-(\un{0},f_2,f_3,f_0)z +
(\un{0},f_3,f_0,\un{0})z^2-(\un{0},f_0,\un{0},\un{0})z^3$ and hence
$G[z]$ is a noncatastrophic generator for a $(5,4,4)$ convolutional
code. The GSB for such a code is $7$. The control matrix for the code
is $H[z]=f_4-f_1z+f_2z^2-f_fz^3+f_0z^4$. The GBS is not
attained. However this code can correct errors
of the form $\al z^i$ with $\al\in
F^5$ and errors of the form $\al[z] = \al_0 z^{i_0} + \al_1z^{i_1} + \ldots + \al_rz^{i_r}$,
where $i_{j+1} \geq i_j + 4 $ for $0 \leq j \leq r-1$, and $\al_i\in
F^5$. In other words, provided the errors occur in vectors at least 
degree $5$ apart then the control matrix will correct these
algebraically. 
\item\label{high4} A code similar to example \ref{high3} with MDS can
  be made as follows. Let $U,V$ be as in example \ref{high3}. 
 Define $G[z]= \begin{pmatrix} e_0 \\ e_1 \\ e_2 \\ e_3 \end{pmatrix}
+ \begin{pmatrix} e_4 \\ e_3 \\ e_1 \\ \un{0} \end{pmatrix}z$.
The right inverse for $G[z]$ is $K[z] =
(f_0,f_1,f_2,f_3)-(\un{0},f_2,\un{0},f_0)z +
(\un{0},\un{0},\un{0},f_2)z^2$ and hence
$G[z]$ is a noncatastrophic generator for a $(5,4,3)$ convolutional
code. The GSB for such a code is, by\cite{ros}, $5$.
The control matrix is $H[z]=f_4-f_0z$. Then provided $(e_1,e_2,e_3)$
generates a $(5,3,3)$ linear code, $G[z]$ can be shown to be an MDS
$(5,4,3;1,5)$ convolutional code.   
\item The process in the previous examples can be continued so as to
  construct high rate convolutional codes with good error correcting
  properties. The codes will correct a single vector in error or a
  series of such errors provided they are sufficiently apart. 
\end{enumerate}

 
\subsection{Algebraic decoding techniques}\label{decoding}
Here we give algebraic decoding methods for examples in
Section \ref{examples1}. Note that these are prototype examples and the decoding techniques themselves can be
applied to similar higher length constructions. 

The following Lemma is straight forward but is useful for the algebraic decoding techniques.
\begin{lemma}\label{refdecod}
Let the columns of an $n\ti n$ invertible matrix $A$ be denoted by $f_0,
f_1, \ldots, f_{n-1}$. Suppose $\al$ is a $1\ti n$ vector and that
the scalars $\al f_i$ are known for $0\leq i \leq n-1$. Then the
vector $\al$ may be determined from the inverse of $A$.
\end{lemma}
\begin{proof} It is seen that $\al(f_0,f_1,\ldots, f_{n-1})$ is known
  and thus $\al A$ is known. But $A$ is invertible and so $\al$ is
  obtainable from the inverse of $A$.
\end{proof}  

Notice that the inverse of $A$ may  already be known when the code is
constructed. 

This technique is useful when all the columns of a nonsingular matrix appear in the control matrix. 

Consider now the code in Section \ref{examples1},  example  \ref{item2} which has
check matrix $H=f_2-f_1z + f_0z^2$. Assume
now an error in transmission has the form $\al $ or more
generally of the form $\al z^i$ for an unknown $\al \in F^3$. Then
$\al H$ is known from which it is deduced that that the $\al f_i$ are known
for $i=0,1,2$. Hence by Lemma \ref{refdecod} $\al$ is known from the
inverse of $(f_0,f_1,f_2)$ which is  $U=\begin{pmatrix} e_0 \\
 e_1 \\ e_2  \end{pmatrix}$. It may also be shown that an error of the
form  $\al[z] = \al_0 z^{i_0} + \al_1z^{i_1} + \ldots + \al_rz^{i_r}$,
where $i_{j+1} \geq i_j + 3 $ for $0 \leq j \leq r-1$,   can be
corrected.  

Consider now the code in Section \ref{examples1}, Item \ref{item1}, 
where the control matrix is $H[z] = (f_1,f_2)
- (f_0,\un{0})z -(\un{0}, f_0)z^2$. We show that this can correct, by an algebraic
method, errors of the form $\al[z]= \al z^i + \be z^j$. 
\begin{lemma} Suppose $\al[z]H[z]$ is known with $\al[z]= \al z^i + \be
 z^j$. Then $\al[z]$ can be determined. 
\end{lemma}
\begin{proof} We do this for $\al[z] = \al + \be z$. The more general
  case is similar. 
Looking at the coefficients of $z^0,z^1, z^2,z^3$ in $\al[z]H[z]$ in turn, then 
(i) $\al f_1, \al f_2$ are known, (ii) $\al f_0 + \be f_1, \be f_2$ are known, (iii) $\be
 f_0, \al f_0$ are known, (iv) $\be f_0$ is known. From this it is seen that $\al f_0, \al
 f_1, \al f_2$ and $\be f_0, \be f_1, \be f_2$ are known 
from which by Lemma \ref{refdecod} $\al$ and $\be$ can be  determined. Note 
 $\be f_0$ is determined in a different way and this is an extra check.  

\end{proof}   

Consider now example  \ref{example5} in which the control matrix is 
$H[z]=(f_3,f_4)-(f_1,f_2)z+(\un{0},f_0)z^2$. If the error is of form $\al z^i$
then $\al z^iH[z]$ determines $\al f_3,\al f_4, \al f_1,\al f_2,\al f_0$
and so by Lemma \ref{refdecod} $\al z^i$ is determined. More generally errors of the form $\al + \be z^i$ for $i\geq 3$ are
determined or more generally errors of the form $\al z^i + \be z^j$ with
$j\geq i + 3$ are determined. Errors of the form $\al + \be z^2$ are
also determined provided either $\al$ or $\be$ have at most two non-zero
entries and any row of $U$ generates a $(5,1,5)$ (which is then two-error
correcting) linear code.



\section{Specify rate and errors}\label{specify}

Suppose it is required to construct a convolutional code with rate $R$
and which can correct $t$ errors at each vector  component.

To correct $t$ errors at each vector component specifically means that if $\al_0 + \al_1z+
\ldots $ is the error vector then up to $t$ errors may be corrected
at each $\al_i$.

Here convolutional codes of memory 1 satisfying these
conditions are shown to exist with efficient decoding algorithms and with maximum possible distance.

Let $R=\frac{r}{n}$ be the required rate and suppose it is
required to correct $t$ errors at each component. Require $n-r =
2t$. Suppose now $n-r$ is even; the case where $n-r$ is odd is
similarly dealt with. As $r=nR$ require $n(1-R) = 2t$ which implies
$n= \frac{2t}{1-R}$.

Let $R=\frac{r}{n}$ and require $n\geq \frac{2t}{1-R}$.

Let $U$ be an $n\ti n$ matrix defined by methods in \cite{hurleyunit}
with the property that rows taken in succession\footnote{There are
  other possibilities, see \cite{hurleyunit} and \cite{hurley} for more generality} generate an mds linear code; for example when
$U$ is a Fourier matrix,  this is the case see \cite{hurleyunit}. Now
$U$ has the form $U= \begin{pmatrix}e_0 \\ e_1 \\ \vdots
  \\ e_{n-1}\end{pmatrix}$ for rows $e_i, i= 0, 1, \ldots, n-1$.
 
Suppose first of all that $n> 2r$.  Define $A[z]= \begin{pmatrix}
  e_0 \\ e_1 \\ \vdots \\ e_{r-1}\end{pmatrix} +
 \begin{pmatrix} e_{r} \\ \vdots \\ e_{n-1} \\ \un{0} \\ \vdots \\ \un{0}\end{pmatrix}z$ 

Note that enough zero vectors are added at the end of $e_{n-1}$ so that the coefficient of $z$ also has $r$ rows.  
Then $A[z]$ defines a rate $\frac{r}{n}$ convolutional code of memory
1 and degree $n-r$; it may be called a {\em partial} unit-memory code 
in some quarters. 

By \cite{ros} the maximum free distance for a $(n,r,\de: \mu)$ convolutional
code is $(n-r)(\floor{\frac{\de}{r}}+ 1) + \de + 1$.  For the case here
this is $(n-r) + (n-r) +1 = 2(n-r)+1$. It is now shown  that this is 
the free distance attained.  Now

$\begin{pmatrix} e_0 \\ e_1 \\ \vdots \\ e_{r-1}\end{pmatrix} +
 \begin{pmatrix} e_{r} \\ \vdots \\ e_{n-1} \\ \un{0} \\ \vdots \\ \un{0}\end{pmatrix}z * ((f_r,f_{r+1}, \ldots,f_{n-1}) - (f_0, f_1,\ldots, f_{n-r-1})z) = 0_{r\ti (n-r)}$.

Thus $(f_r,f_{r+1}, \ldots,f_{n-1}) - (f_0, f_1,\ldots, f_{n-r-1})z$
is a control matrix for the convolutional code with $A[z]$ as
generator matrix.

Also $\begin{pmatrix} e_0 \\ e_1 \\ \vdots \\ e_{r-1}\end{pmatrix} +
 \begin{pmatrix} e_{r} \\ \vdots \\ e_{n-1} \\ \un{0} \\ \vdots \\ \un{0}\end{pmatrix}z * (f_0,f_{1}, \ldots,f_{r-1}) = I_{r\ti (r)}$,

and so $A[z]$ has a right inverse and hence is a non-catastophic
generator matrix of a convolutional code.

It will also correct $\floor{\frac{n-r}{2}}=t$ errors at each 
component of the error vector.
 This is illustrated by an example.

Let $R=\frac{7}{11}$ and it is required to correct $2$ errors on each
vector component of the error vector. Consider

$A[z]= \begin{pmatrix} e_0 \\ e_1 \\ e_2 \\ e_3 \\ e_4 \\ e_5  \\ e_{6}\end{pmatrix} +
 \begin{pmatrix} e_{7} \\ e_{8} \\ e_9 \\ e_{10} \\ \un{0} \\ \un{0} \\ \un{0} \end{pmatrix}z$.

The control matrix is $H[z]= (f_7,f_8,f_9,f_{10}) -
(f_0,f_1,f_2,f_3)z$.





Suppose now $E[z]= \un{\al}_0+ \un{\al}_1 z + \un{\al}_2z^2 + \ldots$
is an error vector where $\un{\al}_i$ are $11 \ti 11$ unknown vectors.

Now look at $E[z]H[z]$. The coefficient of $z^0$ gives that
$\un{\al}_0(f_7,f_8,f_9,f_{10})= \un{0}$. Now $(f_7,f_8,f_8,f_{10})$
is the check matrix of an $[11,7,5]$ linear code so if $\un{\al}_0$
has $\leq 2$ errors these can be located and determined. 

Having determined $\un{\al}_0$, now  by looking at the coefficient of $z$ and assuming
$\un{\al}_1$ has at most 2 errors then $\un{\al}_1$ can be determined.
Then proceed to find $\un{\al}_{i}$ once $\un{\al}_{i-1}$ has been
determined for $i>2$.

The error correcting method for the linear part has been determined in
\cite{hurleyunit}.

It is now  shown that the code has free distance $9$.   Let $G_0
= \begin{pmatrix} e_0 \\ e_1 \\ e_2 \\ e_3 \\ e_4 \\e_5 \\
  e_{6}\end{pmatrix}; \, G_1= 
 \begin{pmatrix} e_{7} \\ e_{8} \\ e_9 \\ e_{10} \\ \un{0} \\ \un{0} \\ \un{0} \end{pmatrix}$.

A codeword is $(\un{\al}_0+\un{\al}_1z+\un{\al}_2z^2+ ...)(G_0+G_1z)$.

Let $\un{\al}_0=(\al_1,\al_2, \ldots, \al_7)$.  Now $\un{\al}_0G_0$ has distance
$\geq 5$.  If $\un{\al}_0$ is not of the form $(0,0,0,0, *,*,*)$ then
$\un{\al}_0G_1$ has distance $\geq 8$ and hence $\un{\al}_0(G_0+G_1z)$ has free
distance $\geq 9$.  Suppose then $\un{\al}_0 = (0,0,0,0,*,*,*)$. Then
$\un{\al}_0G_0$ is a nontrivial sum of $ \{e_3,e_4,e_5\}$ and so has distance
$\geq 9$ as $\{e_3,e_4,e_5\}$ generate an $(11,3,9)$ code.

Similarly consider $(\un{\al}_0+\un{\al}_1 z + \ldots +
\un{\al}_kz^k)(G_0+G_1z)$. The coefficient of $z^0$ in this product 
 has distance $\geq 5$
and the coefficient of $z^{k+1}$ has distance $\geq 8$ unless $\un{\al}_k=
(0,0,0,0,*,*,*)$ in which case the coefficient of $z^k$ has distance
$\geq 9$. Thus the free distance is $9$ which is the maximum
attainable for such a code, \cite{ros}.

\section{Further Fourier/Vandermonde units}\label{cheb1}

More prototype samples derived from Fourier/Vandermonde  unit schemes 
are given here. 
 
Let  $U$ be a Fourier $n\ti n$ matrix over a field $\F$ and $UV=I_n$.
Then as shown in \cite{hurleyunit} any matrix formed by taking $r$ rows
in succession or in arithmetic sequence with difference $k$ satisfying
$\gcd(n,k)=1$   generates an  $(n,r)$ mds linear code. 

Using  such $U,V$ 
with which to  construct 
convolutional codes by the method of  Section \ref{rows} 
 will give good free distances and efficient decoding algorithms. 
 The free distances and/or lower bounds on the free distances 
may  often be proven  algebraically 
and the codes are relatively easy to implement and simulate.  

Suppose  a primitive  $n^{th} $ root of unity, $\al$, exists in a field
$K$.
The Fourier $n\ti n$ matrix $F_n$ over $K$ is

$F_n=\begin{pmatrix} 1& 1& 1 & \ldots & 1 \\ 1 & \al & \al^2 & \ldots
& \al^{(n-1)} \\ 1 & \al^2 & \al^4 & \ldots & \al^{2(n-1)} \\ \vdots &
\vdots & \vdots & \vdots & \vdots \\ 1 & \al^{n-1}& \al^{2(n-1)} &
\ldots & \al^{(n-1)(n-1)} \end{pmatrix}$.

As $n$ must then divide $(|K|-1)$, the inverse of $\al$ exists in $K$
and is easily determined. 

$F_n^*=\begin{pmatrix} 1& 1& 1 & \ldots & 1 \\ 1 & \al^{-1} &
  \al^{-2} & \ldots & \al^{-(n-1)} \\ 1 & \al^{-2} & \al^{-4} & \ldots
  & \al^{-2(n-1)} \\ \vdots & \vdots & \vdots & \vdots & \vdots \\ 1 &
  \al^{-(n-1)}& \al^{-2(n-1)} & \ldots &
  \al^{-(n-1)(n-1)} \end{pmatrix} $



Suppose the Fourier matrix $F$ over some field (for which it exists) has rows
$\{e_0,e_2,\ldots, e_{n-1}\}$. Then in general denote the inverse of $F$ by
$F^*$ and the columns of $F^*$ by $\{e_0^*, e_1^*, \ldots, e_{n-1}^*\}$.

\begin{theorem}\label{distancemds} Suppose $U$ is an $n\ti n$ matrix with rows $\{e_0, e_1, \ldots,
 e_{n-1}\}$ with $UV=I$ where $U$ is a Fourier matrix over a field $\F$. 
 Then the (linear) code
 generated by any $r$ of the rows of $U$ in succession or in arithmetic sequence $k$ with $\gcd(n,k)=1$ is an mds $(n,r,n-r+1)$ linear code.
\end{theorem}
\begin{corollary}\label{distance} A non-zero linear combination of
 $r$
  rows of $U$ has 
 distance (= number of non-zero entries) greater than or equal to $n-r+1$.
\end{corollary}



\subsection{Prototypes}\label{design} 
\subsubsection{Length $3$}
Suppose the Fourier $3\ti 3$ matrix $F_3$ exists over $K$. 
Cases of such $K$ are $GF(2^2)$, $GF(5^2)$, $GF(7)=\Z_{7}$. 
 Denote the rows of $F_3$ by $\{e_0,e_1,e_2\}$. Define
$G[z]=e_0+e_1z+e_2z^2$. Then $G[z]$ is a noncatastrophic generator
matrix for a $(3,1,3;3,9)$ convolutional code which is mds. For
example when $K=\Z_7$ then $2$ is a primitive $3^{rd}$ root of $1$. and
get $e_0=(1,1,1), e_1= (1,2,4), e_2=(1,4,2)$ with entries which are integers
modulo $7$. 

\subsubsection{Low rate}  
Take $s=n-1$ in equation (\ref{eq:geneq}) of
Section \ref{rows} 
and  use all the rows of
a Fourier  matrix $F_n$. 
Define $$G[z] = e_0+e_1z+\ldots + e_{n-1}z^{n-1}$$\label{test} where $F_n$ has rows
$e_i$, for $0\leq i \leq n-1$. 
This is a $(n,1,n-1;n-1)$
convolutional code with degree and memory $(n-1)$. The rate is not
very good particularly for large $n$  but indeed the maximum free
distance is attained.  The maximal free distance
attainable by  such a  code  is by \cite{ros}  
$(n-r)(\floor{\de/r}+1)+ \de +1 = (n-1)(n-1+1)+ n-1+1= n^2$. 

\begin{proposition}\label{general1} The free distance of the code generated by
 $G[z]$ above is $n^2$.
\end{proposition}
The proof is omitted. Note that each $e_i$ generates a $(n,1,n)$ code
and that any non-zero combination of $r$ of the $e_i$ has distance
$\geq (n-r+1)$. 

By Proposition \ref{jerome} $G[z]$ is noncatastrophic.
 It is fairly easy to show this directly and we give  an
independent proof and produce  the control/check matrix. 


\begin{lemma}\label{gen2} 

Let  $G[z] = e_0+e_1z+\ldots +
  e_{n-1}z^{n-1}$ as above.  

(i) Define  $H(z) = e_0^*$. Then $G[z]H(z) = 1$.

(ii) Define $K[z]= (e_1^*,e_2^*, \ldots,e_{n-1}^*) - (e_0^*, \un{0},
  \ldots, \un{0}) - (\un{0},e_0^*, \un{0}, \ldots, \un{0})z - \ldots -
  (\un{0}, \un{0}, \ldots, \un{0}, e_0^*)z^{n-1}$ where $\un{0}$ is
  the $n\ti 1$ zero matrix. Then $G[z]K[z] =
  0_{1\ti (n-1)}$ and $K[z]$ has rank $n-1$.
\end{lemma}
\begin{proof} These follow by direct multiplication on noting 
$e_ie_j^*= \delta_{ij}$  (Kronecker delta).

\end{proof}
\begin{corollary}\label{gen1} $G[z]$ 
  is noncatastrophic. 
\end{corollary}


\subsubsection{Length $5$}

Let $F_5$ be a Fourier matrix over a field  $K$. 

In such a field let $\om$ be a primitive $5^{th}$ root of unity.  

Consider $GF(2^4)$.
Here note that $x^4+x^3+x^2+x+1$ is irreducible over $\Z_2$
and if $\al$ is a primitive element then choose $\om=\al^3$. 

In $GF(11)$ it is seen that the order of $2$ is $10$ 
with $2^{10}=1$ 
and let $\om = 2^2=4$ to get a primitive $5^{th}$ root
of unity.

In all cases 
$F_5= \begin{pmatrix}1&1&1&1&1 \\ 1&\om & \om^2 & \om^3 & \om ^ 4 \\ 1
  & \om^2& \om^4 &\om & \om ^3 \\ 1 & \om^3&\om & \om^4&\om \\ 1 &
  \om^4 & \om^3&\om^2&\om \end{pmatrix}$ 

When for example 
$K=GF(11)$ then  
$F_5 = \begin{pmatrix}1&1&1&1&1 \\ 1 & 4 & 4^2 &4^3&4^4 \\ 1 & 4^2
  &4^4 & 4&4^3 \\ 1 & 4^3 & 4& 4^4&4^2 \\ 1 & 4^4&4^3&4^2&4
       \end{pmatrix}=  \begin{pmatrix} 1&1&1&1&1 \\ 1&4 & 5 & 9& 3 \\ 1
& 5 & 3&4& 9 \\
	1& 9&4&3&5 \\ 1&3&9&5&4\end{pmatrix}$

is a Fourier matrix over $GF(11)$. 
The entries are elements of $\Z_{11}$. 



Let the rows of $F_5$ be denoted by $\{ e_0,e_1,e_2,e_3,e_4\}$.

By Proposition \ref{general1} the code generated by 
$G[z]= e_0+e_1z+e_2z^2+e_3z^3+ e_4z^4$ is an  
 GSB $(5,1,4;4,25)$ convolutional code and $G[z]$ is noncatastrophic.

Consider $E_0= \begin{pmatrix}e_0\\ e_1\end{pmatrix},
E_1=\begin{pmatrix}e_2 \\ e_3 \end{pmatrix} $ and let 
$G[z]= E_0+E_1z$. This is a $(5,2)$ with $\delta=2,\mu=1$ which can have
at best free distance $(n-r)(\delta/r+1)+\delta +1= 3(2)+2+1=9$ by \cite{ros}.
\begin{proposition}
(i) The matrix $G[z]$ is noncatastrophic.

(ii) The code generated by $G[z]$ has $d_{free}= 8$.

(iii) For any inputted word $f(z)$ of support $\geq 2$ the codeword  
 $f(z)G[z]$ has distance $\geq 10$.

\end{proposition}

The proof is omitted. 


\subsubsection{Length $7$}

Consider a Fourier $7 \ti 7$ matrix $F_7$ over a field $K$. 
Denote the rows of $F_7$ by $\{e_0,e_1, \ldots,
e_6\}$. 

Setting $G[z]= e_0+e_1z+e_2z^2+ \ldots + e_6z^6$ gives a $(7,1,6;6,
49)$ by Theorem \ref{general1}. 
Setting $E_i=\begin{pmatrix} e_i \\ e_{i+1} \end{pmatrix}$ for
$i=0,1,2$ and then letting $G[z]= E_0+E_1z+E_2z^2$ gives a
  $(7,2,4;2,18)$ code. The maximum free
  distance of such a $(7, 2, 4;2)$ code is $(5)(4/2+1)+4+1=20$. 

Suppose we wish to obtain a memory $3$ code from this structure. Let
$E_0= \begin{pmatrix} e_0 \\ e_1 \end{pmatrix}, E_1= \begin{pmatrix}
  e_2 \\ e_3 \end{pmatrix}, E_2 = \begin{pmatrix} e_4
  \\ e_5 \end{pmatrix},
E_3= \begin{pmatrix} e_5 \\ e_6 \end{pmatrix}$. Note here that $E_2,
E_3$ have $ e_5$ as common. Set $G[z] = E_0+E_1z+E_2z^2+E_3z^3$. Then
$G[z]$ is noncatastrophic follows  as it is easily
seen that $G[z]E_1^*= I_{2\ti 2}$. 
It may be shown that this is a
$(7,2,6;3, 24)$ code. The maximum free distance of such a $(7,2,6;3)$
code is by \cite{ros} $(5)(6/2+1)+6+1 = 27$.

Setting $E_0=\begin{pmatrix} e_0 \\ e_1 \\ e_2 \end{pmatrix},
E_1=\begin{pmatrix} e_3 \\ e_4 \\ e_5 \end{pmatrix}$ and $G[z] =
E_0+E_1z$ gives a noncatastrophic matrix generator of a $(7,3,3;1,10)$
code.
The maximum free distance of such a code is $4(3/3+1)+3+1= 12$. Also
note that $f(z)G[z]$ has distance $\geq 12$ for any $f(z)$ of support
$\geq 2$. Permuting the $e_i$ gives $7!$ such codes.  

Setting $E_0=\begin{pmatrix} e_0 \\ e_1 \\ e_2 \end{pmatrix},
E_1=\begin{pmatrix} e_3 \\ e_4 \\ e_5 \end{pmatrix},
E_2=\begin{pmatrix} e_4 \\ e_5 \\ e_6 \end{pmatrix}$ and then $G[z] =
E_0+E_1z+E_2z^2$, gives a $(7,3,6;2)$ code. $G[z]$ is noncatastrophic
as may be verified directly. 
The free distance is $15$. 
The maximum free distance
of such a code is $(4)(6/3+1)+6+1=17$.   
\subsubsection{$GF(23)= \Z_{23}$}
The next case taken is that of a Fourier matrix over a finite field in
which the entries are elements of $\Z_{23}$.

  In $GF(23)$ a primitive element is $5$ and so $5^2=2 (\mod 23)$ is
  an element of order $11$ from which the Fourier matrix $F_{11}$ over
  $GF(23)$ can be constructed. This gives

$F_{11} =\begin{pmatrix} 1 & 1 &1 & \ldots &1 \\ 1 & 2 & 2^2 & \ldots
  & 2^{10} \\ 1 & 2^2& 2^4 &\ldots & 2^{20} \\ \vdots & \vdots &
  \vdots &\vdots &\vdots \\ 1 & 2^{10}& 2^{20} & \ldots &
  2^{100} \end{pmatrix} =\begin{pmatrix} 1 & 1 &1 & \ldots &1 \\ 1 & 2
  & 4 & \ldots & 12 \\ 1 & 4& 14 &\ldots & 6 \\ \vdots & \vdots &
  \vdots &\vdots &\vdots \\ 1 & 12& 6 & \ldots & 2 \end{pmatrix}$.

  
Now $11$ is a
  Germain prime with safe prime $23=11\ti 2 +1$ and 
 the Fourier matrix $F_{11}$ over $\Z_{23}$ exists. 
  Chebotar\"ev property.

Denote the rows of the Fourier matrix by $\{e_0,e_1,\ldots, e_{10}\}$. 

Then 
\begin{enumerate}
\item $G[z]= e_0+e_1z+ \ldots + e_{10}z^{10}$ is a noncatastrophic matrix
      for a $(11,1)$ code which has the mds free distance $11^2=121$. There
      are $11!$ similar such codes obtained by permuting the order of
      $0,1,\ldots,10$. 
\item Let $E_0= \begin{pmatrix}e_0 \\ e_1 \\ e_2 \\ e_3 \\ e_4
		\end{pmatrix}, E_1=\begin{pmatrix}e_5 \\ e_6 \\ e_7 \\
				    e_8 \\ e_9 \end{pmatrix}$.

Define $G[z]=E_0+E_1z$. Then $G[z]$ is a noncatastrophic generator
      matrix for a $(11,5, 5)$ convolutional code which has free
      distance $14$. The maximum distance for such a $(11,5,5)$
      convolutional code is $(n-r)(\de/r+1)+\de +1 = (6)(2)+5+1=18$. An
      inputted word of memory $\geq 1$ has distance at least $16$. An 
      inputted word of memory $\geq 2$ has distance at least $18$.
\item Define  $E_0= \begin{pmatrix}e_0 \\ e_1 \\ e_2 \end{pmatrix},
    E_1 = \begin{pmatrix} e_3 \\ e_4 \\ e_5
		\end{pmatrix}, E_2=\begin{pmatrix}e_6 \\ e_7 \\
				    e_8 \end{pmatrix}$.
Define $G[z]= E_0+E_1z+E_2z^2$. Then $G[z]$ is a noncatastrophic
generator matrix for a $(11,3, 6)$ convolutional code. Its free distance
is $27$. The maximum free distance for such an $(11,3,6)$ code is
$(n-r)(\de/r+1)+\de+1 = 8(6/3+1)+6+1=31$.
Now $11!$ such codes may be constructed by permuting the order of
$0,1,2,\ldots, 10$.
 \item Define $E_i=\begin{pmatrix} e_i \\ e_{i+1} \end{pmatrix}$ for
       $i=0,1,2,3,4 $. Define $G[z]=
       E_0+E_1z+E_2z^2+E_3z^3+E_4z^4$. Then $G[z]$ is a noncatastrophic
       generator matrix for a $(11,2,8)$ code and it has free distance
       $50$. The maximum distance for such a $(11,2,8)$ code is
       $(9)(8/2)+8+1=54$. If the inputted word has memory $\geq 1$ then
       the distance is $52$. 
\end{enumerate}







Let $G[z] = E_0+E_1z+E_2z^2$. The free distance of the code generated by
$G[z]$  is $27$. 
Note that a non-zero linear combination of any three of $\{e_0,
\ldots, e_{10}\}$ has distance at least $9$, a non-zero linear combination of
any $6$ of these has distance at least $6$ and a non-zero linear
combination of $9$ of these has distance at least $3$ by Corollary
\ref{distance}. 


\subsection{Series of  examples}\label{examples}
The following are examples of the types of series of convolutional
codes 
that can be constructed using Fourier or Vandermonde matrices. 
\begin{enumerate}
\item Length $3$: Series of MDS $(3,1,2;2, 9)$ convolutional
  codes. Series of $(3,2,2;1,4)$ codes in which codewords for which
  the information vector has  support $\geq 2$ have distance $\geq 5$; 
the GSB here is $5$. 
The
  fields here can be $GF(2^2), GF(2^4)$ or $GF(7)=\Z_7$ 
see \cite{hurley}.
\item Length $5$:  \begin{enumerate} 
\item Series of $(5,2,2;1,8)$ codes. The GSB for such codes is
  $9$. The distance of a codeword which has information vector  of support
  $\geq 2$ is $\geq 10$.  
\item Series of $(5,2,4;2,\geq 12)$ codes. The GSB for such codes is
  $14$. \item Series of $(5,3,3;1, \geq 6)$ codes. The GSB for such
  codes is $8$; a codeword from an information vector  of support $\geq 2$
  has distance $\geq 8$. \item 
Series of $(5,1,4;4,25)$ codes which attain the GSB. \end{enumerate}
The fields here can be $GF(11)=\Z_{11}$ or $GF(7^4)$,
  see \cite{hurley}.
\item Length $7$:  \begin{enumerate} 
\item Series of $(7,3,3;1,\geq 10)$ codes. The GSB for such codes is
      $12$. An inputted word of support $\geq t$ has corresponding
      codeword of distance $\geq 10 + 2(t-1)$. 
\item Series of $(7,2,4;2, 18)$ codes. The GSB for such codes is $20$.
\item Series of $(7,2, 6;3,24)$ codes. The GSB for such codes is $27$.
\item Series of $(7,1,6;6,49)$ codes. The GBS is attained but  
  the rate is not good. 
\end{enumerate}
The fields here can be $GF(3^6)$ or $GF(5^6)$, see
  \cite{hurley}.
\item Length $11$: \begin{enumerate} \item 
Series of $(11,5,5;1,\geq 14)$ codes; the GSB for
      such codes is $18$. Information vectors  of support $\geq t$ have
		    distance $\geq 14 + 2(t-1)$.
\item Series of 
$(11,3,6;2,\geq 29)$ codes. Series of $(11,4,4;1,16)$ codes; the GSB
      here is $19$. \item Series of $(11,4,8;2,24)$ codes; the GSB for such
      codes is $30$.  \item Series of
      $(11,5,10;2,\geq 20)$ codes; the GSB for such is $29$. \end{enumerate} 
\item \vdots
 General $n$ with memory $1$: 
Suppose the Fourier matrix $F_n$ exists over $K$. 
      
Let $k= \floor{n/2}$. Then series of
  $(n,k,k;1,\geq n+3)$ codes are constructed. An inputted word of
      support $\geq t$ has codeword of distance $n+3 +2(t-1)$. 
The GSB for such $(n,k,k;1)$ codes is
  $3\lceil{n/2}\rceil$. The rate is approximately a half. 
\item Memory 1 codes of rate approximately  
1/2 with good free distances may be formed 
from the rows $\{e_0, e_1, \ldots, e_{n-1}\}$ of a matrix.
Set $r= \floor{\frac{n}{2}}$. Let $E_0$ be an $r\ti n$ matrix formed
using  $r$
rows of $\{e_0, e_1, \ldots e_{n-1}\}$ and let $E_1$ be an $r\ti n$
matrix formed using  $r$ other different rows of $\{e_0,e_1, \ldots,
e_{n-1}\}$. 
Form $G[z]= E_0+E_1z$. Then $G[z]$ is
the generator 
matrix of a noncatastrophic convolutional code of type $(n,r)$. Its rank
 and its degree $\de$ and memory $\mu$ are  $r$. The GSB  of such a code is $(n-r)(\de/r+1)+
\de + 1 = (n-r)2+r+1=2n-2r+r+1=2n-r+1 = 2(n-r+1)+r-1$. It is easily
shown that the code formed has free distance at least $2(n-r+1)$. 

\item  Large examples in modular arithmetic: Use for example Germain
	primes; these are primes $p$ such that $2p+1$ is also a
        prime. 
        For example let $K=GF(227)= \Z_{227}$ and
       $n=113$. Here $113$ is a Germain prime and $2*113+1=227$ is the 
corresponding safe prime. Then explicit
         series of $(113, 56, 56;1, \geq 116)$ codes over $\Z_{227}=GF(227)$
         and others may be designed. 
 
\end{enumerate}  

\section{Convolutional Extension to Hamming}
Here a convolutional unit memory code is constructed which mimics the
Hamming linear code and  can correct one error at each vector 
component of an error vector polynomial.

A generator matrix for the Hamming $(7,4,3)$ code over $\Z_2$ obtained from a circulant matrix is 

$K= \begin{pmatrix} 1 &1 & 0 & 1 & 0 &0 &0 \\ 0&1&1&0&1&0&0 \\ 0&0&1&1&0&1&0 \\ 0&0&0&1&1&0&1 \end{pmatrix} = \begin{pmatrix} e_0 \\ e_1\\ e_2\\e_3 \end{pmatrix} $

where the $e_i$ denote the rows of $K$. This can be extended to a unit matrix 

$L=  \begin{pmatrix} 1 &1 & 0 & 1 & 0 &0 &0 \\ 0&1&1&0&1&0&0 \\ 0&0&1&1&0&1&0 \\ 0&0&0&1&1&0 &1 \\ 1&1&1&0&1&0&0 \\ 0&1&1&1&0&1&0 \\ 0&0&1&1&1&0&1\end{pmatrix} = \begin{pmatrix} e_0 \\ e_1\\ e_2\\ e_3 \\ e_4 \\ e_5 \\ e_6\end{pmatrix} $.

Then $M= \begin{pmatrix} 0&1&0&0&1&0&0 \\ 0&0&1&0&0&1&0 \\ 0&0&0&1&0&0&1 \\ 1&1&1&0&1&1&0 \\ 0&1&1&1&0&1&1 \\ 1&1&0&1&1&1&1 \\ 1&0&0&0&1&0&1 \end{pmatrix} = (f_0,f_1,f_2,f_3,f_4,f_5,f_6) $

is the inverse of $L$ where the $f_i$ are the columns of $M$. 

Now define 
$G[z] = \begin{pmatrix} e_0\\ e_1\\ e_2\\ e_3 \end{pmatrix} + \begin{pmatrix} e_4 \\ e_5\\ e_6 \\ \un{0}\end{pmatrix}z $. 

Then $G[z]$ defines a unit memory convolutional $(7,4,3;1)$ convolutional code. The maximum free distance for these parameters is, by \cite{ros},  $(7-4)1+3+1 = 7$.  

It is easy to verify  that 
$H[z] = (f_4,f_5,f_5) - (f_0,f_1,f_2)z$ is a control matrix. (Note -1 = +1 as we are working over $\Z_2$ here.) A right inverse for $G[z]$ is $(f_0,f_1,f_2,f_3)$.

Now $R[z]G[z]= (\un{\al}_0 + \un{\al}_1z+ \ldots ) G[z]$ is of (free) distance $\geq 6$ unless  $R[z]= \un{\al}_0 = (0,0,0,*)$ in which case the distance is $3$. 

However it can correct any error vector in which there is at most $1$ error at each component. For suppose the error (unknown) vector is $(\un{\al}_0 + \un{\al}_1z+ \ldots )$ . Then 
$(\un{\al}_0 + \un{\al}_1z+ \ldots )((f_4,f_5,f_6) + (f_0,f_1,f_2)z) $ is known.
Looking at the coefficient of $z^0$ means that $\un{\al}_0(f_4,f_5,f_6)
$  is known so a single non-zero entry of  $\un{\al}_0$ may be
determined  as $(f_4,f_5,f_6)$ is the check matrix of a $(7,4,3)$ linear
code. When $\un{\al}_0$ is known, looking at the coefficient of $z$ will
determine any one non-zero entry, at most,  of $\un{\al}_1$.  Then
continuing in this way one non-zero entry  occurring at a component may
be determined.  

The decoding can all be done algebraically. 

We can do better by working over $GF(2^3)$. The order of a primitive
element $\om$ in $GF(2^3)$ is $7$ and thus the Fourier $7\ti 7$ matrix over $GF(2^3)$
exists. 
 Let this  Fourier matrix be denoted by 
$F_7 = \begin{pmatrix} e_0 \\ e_1 \\ \vdots \\ e_6 \end{pmatrix}$ with
rows $e_i$ and let its inverse be denoted by $K= (f_0,f_1, \ldots, f_6)$
with columns $f_i$.  

As noted in \cite{hurleyunit} any $r$ rows of $F_7$ taken in order or in arithmetic sequence generate an mds $(7,r,7-r+1)$ linear code. Consider then 

$G[z] = \begin{pmatrix} e_0 \\ e_1 \\ e_2 \\e_3 \\ e_4 \end{pmatrix} + \begin{pmatrix} e_5 \\ e_6 \\ \un{0} \\ \un{0} \\ \un{0} \end{pmatrix} z $.

Now $G[z]$ generates an $(7,5,2;1)$ convolutional code. The control matrix of the code is $(f_5,f_6) - (f_0,f_1) z$. The right inverse of $G[z]$ is
$(f_0,f_1,f_2,f_3,f_4)$; this can be used to return the original
information vector once the errors have been corrected. Binary
arithmetic can be use for encoding over $GF(2^k)$.  

By \cite{ros} the maximal distance of a convolutional code with parameters $(n,r,\de) $ is $(n-r)(\floor{\frac{\de}{r}}+1) + \de + 1$ which in this case is $ (7-5)(1)+ 2+1 =5$. Now we show that the free distance of code generated by $G[z]$ is actually $5$ but also that it can correct up to one error at each component in  an error vector.

Let $G[z] = G_0 + G_1 z$. Then $\un{\al} (G_0 + G_1z) = \un{\al}G_0 + \un{\al}G_1$. Now $\un{\al}G_0$ has distance $\geq 3$ as $G_0$ generates a $(7,5,3)$ linear code. Also $\un{\al}G_1$ has distance $\geq 6$ unless $\un{\al}$ is of form 
 $(0,0, *,*,*)$  in which case $\un{\al}G_0$ has distance $\geq 5$ as $\begin{pmatrix} e_2 \\ e_3 \\ e_4 \end{pmatrix}$ generates a $(7,3,5)$ linear code. Similarly it may be shown that the distance of any $(\un{\al}_0 + \un{\al}_1z + \un{\al}_2z^2 + \ldots ) (G_0+G_1z)$ is $\geq 5$; in fact the distance increases with the (polynomial) degree of the information vector. 

 It can correct any error vector in which there is at most $1$ error at each component. For suppose the error (unknown) vector is $(\un{\al}_0 + \un{\al}_1z+ \ldots )$ . Then 
$(\un{\al}_0 + \un{\al}_1z+ \ldots )((f_5,f_6) + (f_0,f_1)z)$ is known.
 Looking at the coefficient of $z^0$ gives that $\un{\al}_0(f_5,f_5) $
 is known. Now if  $\un{\al}_0$ has just one non-zero entry it may be
 determined precisely as $(f_5,f_6)$ is the check matrix of a $(7,5,3)$
 linear code. In fact the decoding method here is easy as  $\un{\al}_0$
 has one non-zero entry if and only if it a a multiple of a uniquely
 defined row of $(f_5,f_6)$. When $\un{\al}_0$ is sorted, looking at the
 coefficient of $z$ determines $\al_1$ as long as it has just one
 non-zero entry and so on.

\section{Block constructions}\label{blocks} The methods of Section
\ref{rows}
which use  {\em rows} of unit schemes  to 
construct convolutional codes are now  generalised by  using {\em blocks} of
 unit schemes. This  leads in particular to the algebraic construction
of {\em  LDPC convolutional
codes} and {\em self-dual} and {\em dual-containing} convolutional
codes.


Some of the cases here overlap some of those in \cite{jessica}.  

We begin with an illustrative example. Let $U,V$ be
$2n\ti 2n$ matrices with $UV=1$ and $U,V$ have block representations as
follows: $U= \begin{pmatrix} A \\ B \end{pmatrix}, V= (C,D)$ where
$A,B$ are block $n\ti 2n$ matrices and $C,D$ are block $2n\ti n$
matrices.
Then $UV=1$ implies $  \begin{pmatrix} A \\ B \end{pmatrix}(C,D)
= \begin{pmatrix} AC & AD \\ BC & BD \end{pmatrix} = I_{2n\ti 2n}$ and
  hence

\noindent $AC=I_{n\ti n}, AD = 0_{n\ti n}, BC= I_{n\ti n},
BD=0_{n\ti 2n}$.

Define $G[z] = A + Bz$; this is a $n\ti 2n$ matrix. Then

$G[z]C=I_{n\ti n}, G[z](D-Cz)= 0_{n\ti n}$

Hence $G[z]$ is of rank $n$ with left inverse $C$ and so is a
noncatastrophic generator matrix  for  a
$(2n,n, n;1)$ convolutional code $\mathcal{C}$. Also the check/control
matrix is $(D-Cz)$ 
which is also of rank  $n$ as may be shown by producing a right
inverse for it. Now if $V$ is a low density matrix we have
constructed a low density convolutional code; see section
\ref{ldpc} below for relevant definitions of {\em low density}. Let
$d(W)$ denote the distance of a linear code generated by $W$.
\begin{lemma} $ d_{free}\mathcal{C} \geq d(A)
  +d(B)$. 
\end{lemma}

\subsection{General block method}\label{gblock} 
The general block method is as follows. It is similar in principle 
to the row method 
of Section \ref{rows} but has certain constructions  in mind. 

Suppose that 
$UV=1$ in $F_{sn\ti sn}$ and that $U,V$ have block structures 
$U= \begin{pmatrix} A_1 \\ A_2 \\ \vdots \\ A_s \end{pmatrix}$ and 
$V = (B_1,B_2, \ldots, B_s)$ where the $A_i$ are $n\ti sn$ matrices
and the $B_j$ are $sn \ti n$ matrices. Then $A_iB_j =
\delta_{ij}I_n$. Choose 
any $r$ of these blocks of  $U$ to form an $rn\ti sn$ matrix $E$ which
then has rank $rn$; any order on the blocks may be chosen with which to
construct the $E$. There are $\binom{s}{r}$ ways of choosing $r$ blocks. 
Suppose $\{E_0,E_1,\ldots, E_s\}$ are $s$ such $rn\ti sn$ matrices. Define 
\begin{equation}\label{blocks1} 
G[z]= E_0+E_1z+\ldots + E_sz^s.  
\end{equation}

Consider $G[z]$ to be the generating matrix for a convolutional
code. 

 Say that $A\in E$ if $A$ occurs as a block in forming
 $E$ and $A\not \in E$ if $A$ does not occur as a block in $E$.
 
 The following may be proved in a similar manner to Proposition
 \ref{jerome} in Section \ref{rows}.

\begin{proposition}\label{jerome1} Suppose  $G[z]=E_0+E_1z+\ldots + E_sz^r$ and 
that $A_j\in E_0$ implies
 $A_j\not \in E_i$ for $i\neq 0$. Then $G[z]$ is a noncatastrophic
 generator matrix for a convolutional $(n,r)$ code.  
\end{proposition}  





We illustrate the method for a unit with $3$ blocks. 
Suppose $UV=1$ are $3n\ti 3n$ matrices with $U= \begin{pmatrix} A \\ B \\ C
\end{pmatrix}, V= (D,E,F)$ where $A,B,C$ are $n \ti 3n $ matrices and
$D,E,F$ are $3n\ti n$ matrices. This then gives:

$AD=I_{n\ti n}, AE=0_{n\ti n}, AF=0_{n\ti n}$

$BD= 0_{n\ti n}, BE= I_{n\ti n}, BF = 0_{n\ti n}$

$CD= 0_{n\ti n}, CE= 0_{n\ti n}, CF = I_{n\ti n}$

Define $G[z] = A + Bz+Cz^2$. Then $G[z]D=I_{n\ti n},
G[z](E-Dz)=0$. Thus $G[z]$ is a noncatastrophic generator matrix
 a $(3n,n, 2n;2)$ convolutional code $\mathcal{D}$ with check matrix $(E-Dz)$. 
Let $d(Y,X)$ denotes the distance of the linear
 code with 
  generator matrix $\begin{pmatrix} Y\\ X \end{pmatrix}$.
\begin{proposition} $d_{free}\mathcal{D} \geq \min\{d(A)+ d(A,B) + d(B,C) +
  d(C), d(A)+d(B)+d(C)\}$  .
\end{proposition}


Generate $(3n,2n)$ convolutional codes from the unit system as follows. Define

$G[z] = \begin{pmatrix} A \\ B \end{pmatrix} + \begin{pmatrix}
  B\\ C \end{pmatrix}z$.
\begin{proposition}

(i) $G[z]\begin{pmatrix}(D,E) - (0_{3n\ti 0},
  D)z)\end{pmatrix}  = I_{2n}$. \\ 

(ii) $G[z]\{F-Ez+Dz^2\}= 0_{2n\ti n}$.
\end{proposition}
\begin{corollary} $G[z]$ is a noncatastrophic generator matrix for a
  $(3n,2n,2n;1)$ convolutional code $\mathcal{K}$ with check matrix $H[z] =
  F-Ez+Dz^2$.
\end{corollary}

See below for details on LDPC codes but it's worth mentioning the
following at this stage. 
\begin{corollary} $\mathcal{K}$ is an LDPC convolutional code when $V$
  is a low density matrix. If further $V$ has no short cycles then neither does
  the $\mathcal{K}$.
\end{corollary}

\begin{proposition} $d_{free}\mathcal{K} \geq d(A,B) + d(B,C)$.
\end{proposition}

Use $G[z] = \begin{pmatrix} A \\ B \end{pmatrix}+ \begin{pmatrix} B\\ C
						  \end{pmatrix}z +
						  \begin{pmatrix} C \\ B
						  \end{pmatrix}z^2$
to construct a $(3n,2n,4n;2)$ convolutional code.

\section{LDPC convolutional codes}\label{ldpc}

Suppose now $UV=1$ where $V$ is of {\em low density}. Say $V$ has low
density if and only if $V$ has a small number of elements, compared to
its size, in each row and column. \footnote{For our purposes it is only
necessary to ensure that $V$ has a small number of elements, compared
to its size, in certain columns as the control matrix of the 
convolutional codes are constructed from the columns of $V$.} 
Then convolutional codes constructed
by the 
row method of Section \ref{rows} or the block method of Section \ref{blocks}
must necessarily be LDPC (low density parity check) convolutional
codes. By ensuring that $V$ has no short cycles, which can be done by
methods of \cite{hurley33}, the LDPC convolutional codes constructed
will have no short cycles in their control/check matrices. 
It is known that LDPC codes with no short
cycles in their control/check matrices perform well.  

The paper \cite{hurley33} gives methods for  constructing 
classes of  matrices $U,V$ of arbitrary size over various fields with
$UV=1$ where
$V$ is of low density.  Such matrices 
may for example be obtained, \cite{hurley33},  from group ring elements
$u,v$ with
$uv=1$ in which the support of $v$ as a group ring element 
is small; corresponding matrices
may be obtained by mapping the group ring into a ring of matrices as
per \cite{hur3}. Matrices of arbitrary size and over many fields
including  $GF(2)=\Z_2$ satisfying the conditions may be obtained in
this manner. It may also be ensured in the construction 
 that $V$ has no short cycles thus
ensuring the codes obtained have no short cycles in their control 
matrices.  

Since low density implies the length must comparatively be long, to
actually write out examples explicitly is more difficult
in a research paper  but as shown in \cite{hurley33} many such 
 constructions may be formulated. 

In  
\cite{hurley33} examples are given    
 to  construct industry standards LDPC linear codes. Here 
such an  example is modified to produce  LDPC convolutional codes. 

Consider $\Z_2(C_{204}\ti C_4)$ where $C_{204}$ is generated by $g$ and
$C_4$ is generated by $h$. Set  

 $ v=g^{204-75}+ h(g^{204-13} + g^{204-111} + g^{204-168})
+h^2(g^{204-29} + g^{204-34}+g^{204-170}) + h^3(g^{204-27} +
g^{204-180})$.

The support of $v$ is $9$ and is a low density group ring element. As
shown in \cite{hurley33} $v$ has no short cycles.  Its 
inverse $u$ may be easily found but has large
support and so is not written out. The matrices corresponding to $u,v$
are denoted 
$U,V$ (see \cite{hur3}), which in this case are circulant-by-circulant,
and have the forms
$U=\begin{pmatrix} A_0 \\ A_1\\ A_2\\ A_3\end{pmatrix}, V= ( B_0,
B_1,B_2, B_3)$ where $A_i$ are $204\ti 816$ 
matrices and $B_j$ are  (low density) $816\ti 204$ matrices with
$A_iB_j=\delta_{ij}I_{204\ti 204}$. These blocks may now 
be used to construct LDPC convolutional codes of various types. 

For example $G[z] = A_0 + A_1z + A_2z^2+A_3z^3$ is a noncatastrophic 
 generator matrix for an
LDPC convolutional $(816, 204, 612;3)$   
with no short cycles,  and  control (low density) matrix 
$H[z] = (B_1,B_2,B_3)+ (B_0,\un{0},\un{0})z +
(\un{0}, B_0, \un{0})z^2 + (\un{0},\un{0}, B_0)z^3$ where $\un{0}$ is
the zero $816\ti 204$ matrix. The right inverse of $G[z]$ is  easily
deduced. 

Defining 
 $G[z] =\begin{pmatrix} A_0\\ A_1 \\ A_2 \end{pmatrix} + \begin{pmatrix}
  A_1\\ A_2 \\ A_3 \end{pmatrix}z$ gives an LDPC $(816,612, 612;1)$ code
with control low density matrix $H[z] = B_3+B_2z+B_1z^2+B_0z^3$. The
rate here is $3/4$. 

Defining $G[z] = \begin{pmatrix}
  A_0\\ A_1 \end{pmatrix} + \begin{pmatrix} A_1 \\ A_2 \end{pmatrix}z + 
\begin{pmatrix}
 A_2 \\ A_3 \end{pmatrix}z^2$ gives an LDPC $(816,408, 816;2)$
convolutional code
with control low density matrix $H[z] =
(B_2,B_3)+(B_1,\un{0})z+(\un{0},B_1)z^2+(\un{0},B_0)z^3$. The 
rate  is $1/2$. 

Permutations of $\{0,1,2,3\}$ in $U,V$ may be used 
  for further constructions.  
 
The group ring constructions of \cite{hurley33} allow the construction
of many series of these LDPC convolutional codes. 
\section{Self-dual and dual-containing convolutional codes}\label{selfdual}
It is known that self-dual and dual-containing codes leads to the
construction of quantum codes, see \cite{calderbank} and also
\cite{grassl}. 

\subsection{Self-dual convolutional codes}\label{selfdual1}
A convolutional code with generator matrix $G[z]$ is said to
a self-dual code if its dual `transposed' is equal to itself; we
interpret the transpose of a matrix in the normal way but in addition
interpret the `transpose' of $z$ to be $z^{-1}$.\footnote{This
  interpretation makes great
sense when considering group rings.} 
We say then a convolutional code with
generator matrix $G[z]$ is {\em self-dual} if $G[z]G[z^{-1}]\T=0$. 
 Note that if $H[z]$ is a control
matrix then so is $H[z]z^i$ for any $i$. A self-dual code must
necessarily be a $(2n,n)$ code for some $n$. 

Consider  $UV=1$ with  $UV=  \begin{pmatrix} A \\ B \end{pmatrix}(C,D)
= \begin{pmatrix} AC & AD \\ BC & BD \end{pmatrix} = I_{2n\ti
  2n}$ where  $A,B,C,D$ are blocks of size $n$. Then $AC = I_n, AD=
0_{n\ti n}, BC=0_{n \ti n}, BD=I_n$. 
Hence  $G[z]= A+Bz$ has control  matrix $D-Cz$ as $G[z](D-Cz)=0_{n\ti n}$. 
Suppose now $U$ is an orthogonal matrix so that $C=A\T,
D=B\T$. Then $D-Cz = B\T-A\T z$. Suppose also  the field has
characteristic $2$ 
and in this case
the control matrix is $B\T+A\T z$. Now $G[z^-{1}]\T z =B\T+ A\T z $  
and hence
the code is self-dual.
\begin{lemma} Suppose $U=\begin{pmatrix} A \\ B \end{pmatrix}$ is an
  orthogonal matrix in characteristic $2$. Then the convolutional code
  with generator matrix $G[z]= A+Bz$ is a self-dual code.
\end{lemma} 

Methods are developed in \cite{hurley44} in which to construct
orthogonal matrices of arbitrary sizes 
over various fields
including fields of characteristic $2$. The methods involves
constructing such elements in a group ring and then finding the
corresponding matrices as per the embedding \cite{hur3} of the group ring into
a ring of matrices.  

Here is an example. Consider  $\Z_2C_4$ where $C_4$ is the cyclic group
generated by $a$. Let $u= a+a^2+a^3$. Then $u^2=1$, and $u\T=u$. Thus
taking the corresponding matrix, \cite{hur3},  
 
$U=\begin{pmatrix} 0&1&1&1 \\ 1&0&1&1 \\ \hline 1&1&0&1
\\ 1&1&1&0 \end{pmatrix} = \begin{pmatrix}A \\ B\end{pmatrix}$.

we have  $UU\T=1=I_{4\ti 4}$.

Then $G[z]=A + Bz$ is the generator matrix of a self-dual
convolutional $(4,2,2;1)$ code. It is easy to check that the free
distance of the code is $4$; both $A,B$ are generator matrices for 
$(4,2,2)$ linear codes. 




The next example follows methods developed in \cite{hurley44}.
Consider $\Z_2D_8$ where $D_8 =\langle a,b | a^4=1=b^2, ba=a^{-1}b\rangle$ is
the dihedral group of order $8$. Let $u = 1+ b+ ba$. Then $u^2=1,
u\T=u$. The matrix
of $u$ is 

$U=\begin{pmatrix} 1 & 0& 0&0 & 1&1&0&0 \\ 0&1&0&0&1&0&0&1 \\ 0&0&1&0&
    0&0&1&1 \\ 0&0&0&1 &0&1&1&0 \\ \hline 1&1&0&0 &1&0&0&0 \\ 1&0&0&1&0&1&0&0\\
    0&0&1&1 & 0&0&1&0 \\ 0&1&1&0&0&0&0&1\end{pmatrix} = \begin{pmatrix}
							 A \\ B
							\end{pmatrix}$

and $UU\T=1$.

Then $G[z] = A+Bz$ is the generator matrix for a self-dual convolutional
$(8,4,4;1)$. Its free distance is $6$; each of $A,B$ generates a
$(8,4,3)$ linear code. 

Further examples may be generated similar to methods in \cite{hurley44}
by finding group ring elements with $u^2=1, u\T=u$ and then going over
to the matrix representation as per \cite{hur3}. Using for example
$\Z_2D_{16}$ will lead to a $(16,8,8;1,10)$ self-dual convolutional
code. 

\subsection{Dual-containing convolutional systems}\label{dualcontain} 
A dual-containing code is one whose dual is contained within the
code. A self-dual code is a dual-containing code of rate $1/2$. It is
interesting to construct rate $1/2, 3/4, 7/8$ etc.\ dual-containing
codes as quantum convolutional codes may be constructed from these. 
  
Say $G[z]$ is the generator matrix  for a
dual containing  convolutional code $\C$ if the code generated by
$H[z^{-1}]\T$, where $H[z]$ is a control matrix for $\C$, is contained
in $\C$.

Dual-containing convolutional codes may be constructed as follows. 

Let $U$ be an orthogonal $4n\ti 4n$ 
matrix $UU^*=1$ with $U=\begin{pmatrix} A \\ B\\ C
\\ D\end{pmatrix}$ and $U\T= (A\T,B\T,C\T,D\T)$ for blocks $A,B,C,D$ of
size $n\ti 4n$. 
Define 

$G[z]= \begin{pmatrix} A \\ B \\ C \end{pmatrix} + \begin{pmatrix} B
  \\ C \\ D \end{pmatrix} z$. 

Then for $K[z]= (A\T,B\T,C\T) - (\un{0},A\T,B\T)z +
(\un{0},\un{0},A\T)z^2$ it is seen that $G[z]K[z]= I_{3n}$ and so 
$G[z]$ is a noncatastrophic
generator matrix for a convolutional $(4n,3n,3n;1)$ code. 

Now $G[z]\{ D\T-C\T z + B\T z^2-A\T z^3\} = 0_{3n\ti n}$ as may easily be
  verified. Let $H[z^{-1}] = A\T - B\T z^{-1}+C\T z^{-2}-D\T z^{-3}$. Then
  $H[z^{-1}]z^{3} = D\T -C\T z+B\T z^2-A\T z^3$. Thus $H[z^{-1}]$ is a
  control matrix for the code. To show that the  code with generator
  $G[z]$ is dual containing it is necessary to show that the code
  generated by $H[z]\T $ is contained in  the code $\mathcal{C}$
  generated by  $G[z]$. 

Suppose the  characteristic of the field is $2$. Then  
 $H[z]\T = \{(I_n, \un{0},\un{0}) + (\un{0},\un{0}, I_n)z^2\}G[z]$ 
and hence code generated by $H[z]\T$ is contained in  
$\mathcal{C}$. Thus $\mathcal{C}$ is dual-containing. The rank of the
 code is $3n$ and the rank of the dual is $n$. 

Other dual-containing codes may be obtained by permuting 
$\{A,B,C,D\}$ in $G[z]$. 

Here is a specific example. Consider $u=a+a^4+a^7$ in $\Z_2C_8$ where
$C_8$ is generated by $a$. Then $u^4=1, u^2\neq 1, u^T=u$. Thus $u^2=
a^2+1+ a^6$
is orthogonal matrix and is symmetric. The matrix
of $u^2$ has the form $U=\begin{pmatrix} A \\ B \\ C \\D\end{pmatrix}$
and $U$ is 
orthogonal and symmetric. Then $G[z]= \begin{pmatrix} A \\ B
  \\ C \end{pmatrix} + \begin{pmatrix} B  \\ C \\ D \end{pmatrix} z$
determines  a dual-containing $(8,6,6;1)$ convolutional code. 

Here \\ $A= \begin{pmatrix} 1&0&1&0&0&0&1&0
  \\ 0&1&0&1&0&0&0&1\end{pmatrix}$, 
$B = \begin{pmatrix} 1&0&1&0&1&0&0&0
  \\ 0&1&0&1&0&1&0&0\end{pmatrix}$,\\
$C= \begin{pmatrix} 0&0&1&0&1&0&1&0
  \\ 0&0&0&1&0&1&0&1\end{pmatrix}$,  
$D=\begin{pmatrix} 1&0&0&0&1&0&1&0
  \\ 0&1&0&0&0&1&0&1\end{pmatrix}$.

The free distance is $4$.
\quad 

We may also construct self-dual codes from this set-up. Again 
 $U$ is an orthogonal $4n\ti 4n$  
matrix, $UU^*=1$ with $U=\begin{pmatrix} A \\ B\\ C
\\ D\end{pmatrix}$ and $U\T= (A\T,B\T,C\T,D\T)$ for blocks $A,B,C,D$ of
size $n\ti 4n$.
Define $G[z]=\begin{pmatrix}A \\B \end{pmatrix}+\begin{pmatrix}C
\\ D\end{pmatrix}z$. Then $G[z]\{(C\T,D\T) - (A\T,B\T)z\} = 0_{2n\ti
  2n}$ and $G[z]$ is a generator matrix for a convolutional code which
is self-dual in characteristic $2$. By permuting the $\{A,B,C,D\}$ other
different self-dual convolutional codes may be obtained. Note that
$\{A,B,C,D\}$ must be different as $U$ is invertible and  
this gives $4!$ self-dual convolutional codes.  The distances of the
codes depend on the distances of the linear codes generated by
$\{A,B,C,D\}$.  

Here are examples of dual-containing code in characteristic $3$.
Suppose $U$ with rows $\{e_0,e_1, e_2\}$ is an orthogonal matrix. 
Then $UV=I$ where $V$ has columns $\{e_0\T,e_1\T,e_2\T\}$.
Consider $G[z] = \begin{pmatrix} e_0 \\ e_1 \end{pmatrix}+
\begin{pmatrix}e_1 \\ e_2 \end{pmatrix}z $.
Then $G[z]\{e_2\T- e_1\T+e_0\T z^2\}=0_{2\ti 1}$ and thus $H[z]= e_0-e_1z+e_2z^2$
satisfies $G[z]H[z]\T=0_{2\ti 1}$.
Now $((1,0)+ (0,1)z)G[z]= e_0+e_1z+e_1z+e_2z^2=e_0-e_1z+e_2z^2$ in
characteristic $3$. Thus the code generated by $H[z]$ is contained in
the code generated by $G[z]$ and hence  the code generated by $G[z]$ is
dual-containing. The code is a $(3,2)$ code and the dual is a $(3,1)$
code.

Larger rate dual-containing codes may also be constructed. Here we
indicate how dual-containing convolutional codes 
of rate $7/8$ can be constructed. The process may
be continued for higher rates. 
 Let $U$ be an orthogonal $8n\ti 8n$ 
matrix $UU\T=1$ with $U=\begin{pmatrix} A_0 \\ A_1\\ \vdots \\ A_7
\end{pmatrix}$ and $U\T= (A_0\T,A_1\T,\ldots,A_7\T)$ for blocks $A_i$ of
size $n\ti 8n$. Define 

\begin{equation}\label{kris} G[z] =\begin{pmatrix} A_0 \\ A_1\\ \vdots \\ A_6
\end{pmatrix}+ \begin{pmatrix} A_1\\ A_2 \\ \vdots \\ A_7
\end{pmatrix}z \end{equation}

Then $G[z]\{A_7\T-A_6\T z + \ldots + A_0\T z^7\}= 0_{7n\ti n}$. It may
then be shown that $G[z]$ is a generator matrix for a convolutional
 $(8n,7n,7n;1)$ code which is dual-containing when the characteristic is $2$. 

Specifically For example consider $FC_{16}$, with $F$ of
characteristic $2$ and $C_{16}$ is generated by $a$.
Let $u=a+a^7+a^8+a^9+a^{15}$. Then $u^2=1, u\T=u$. The matrix $U$
corresponding to $u$ as per the isomorphism in \cite{hur3}  is
circulant and has the form $\begin{pmatrix} A_0
  \\ A_1\\ \vdots \\ A_7 \end{pmatrix}$ for $2\ti 16$ matrices $A_i$. 
The resulting $G[z]$ in equation \ref{kris} is a
  dual-containing convolutional $(16,14,14;1,4)$ code; the rate is
  $\frac{7}{8}$. By taking
pairs of the $A_i$ together and forming for example $G[z]=
\begin{pmatrix} A_0
  \\ A_1 \\ \hline A_2 \\  A_3 \\ \hline A_4\\ A_5\end{pmatrix}
+ \begin{pmatrix} A_2\\ A_3 \\ \hline A_4 \\ A_5\\ \hline
  A_6\\ A_7 \end{pmatrix}z$ 
give $(16,12,12;1)$ dual-containing convolutional codes.


\section{Conclusions}\label{attrib}
\begin{itemize} \item Methods are devised for constructing 
  and analysing series of convolutional codes 
      using rows or blocks of invertible
      schemes. These are relatively easy to
  describe and implement.  

\item The right inverse of a noncatastrophic generator matrix of such a code 
and the control matrix are functions of the columns or
blocks of the  inverse of the scheme; these  can then be calculated
directly.   
\item The structures of the control matrices in the forms given lead  
to efficient implementable error-correcting algebraic techniques. 

\item Different convolutional codes, of the same or different rates, 
may  be derived from a single unit scheme type. 
\item Series of good
  convolutional  codes may be designed by using
      unit schemes with special properties.  
\item Codes are constructed to a given rate and given error-capability
      at each component,  and  efficient algebraic decoding algorithms are described for these.
\item The general constructions are embracing and the  explicit
      constructions/examples given, as well as having their own 
intrinsic interest, are an 
      indication of further potential constructions from  the general schemes. 
      Existing constructions 
      occur as special cases of the general construction here. 

\item  Convolutional codes may be decoded by known techniques, such as
       Viterbi decoding or sequential decoding,  
\cite{blahut} Chapter 11, \cite{joh} or \cite{mceliece}, and  algebraic decoding algorithms are
       known to exist for some codes, see
       \cite{heine1}, \cite{rosenthal}. Here efficient algebraic decoding
       algorithms are developed. 

\item Self-dual and dual-containing convolutional codes may  
      be designed and analysed. 
 Dual-containing codes are important
      for the construction of quantum codes, \cite{calderbank}.
LDPC (low density parity
      check) convolutional codes may be  designed and analysed using
       unit schemes;  here the inverse in the unit scheme has `low density' the codes  may be designed so that the control matrices 
      have no short cycles. 
\end{itemize}

\end{document}